\documentclass[a4paper,11pt]{amsart}
\usepackage[margin=3cm]{geometry}
\usepackage{amsfonts,amsmath,amssymb,amscd,graphicx}
\usepackage[colorlinks]{hyperref}
\usepackage{mathrsfs}

\usepackage{tensor}
\usepackage{breqn}

\usepackage{graphicx,xcolor}

\usepackage{caption, subcaption}

\usepackage[english]{babel}
\usepackage[utf8]{inputenc}
\usepackage{ucs} 
\usepackage{float}%
\usepackage{enumitem}
\usepackage{graphicx}
\usepackage{tikz}
\usepackage{indentfirst}
\usepackage{fancyhdr}
\usepackage[all]{xy}
\usepackage{physics}
\usepackage{tikz-cd}
\usepackage{comment}

\theoremstyle{theorem}
\newtheorem{theorem}{Theorem}[section]
\newtheorem{corollary}[theorem]{Corollary}
\newtheorem{lemma}[theorem]{Lemma}

\newtheorem{proposition}[theorem]{Proposition}

\theoremstyle{definition}
\newtheorem{definition}[theorem]{Definition}
\newtheorem{remark}[theorem]{Remark}

\numberwithin{equation}{section}

\newenvironment{preuve}[1][]
{\vskip 2mm  \noindent\emph{\bf Proof#1. }}{$\Box$ \vskip 2mm}

\newcommand{\beq}{\begin{eqnarray}}
\newcommand{\eeq}{\end{eqnarray}}
\newcommand{\beqe}{\begin{eqnarray*}}
	\newcommand{\eeqe}{\end{eqnarray*}}

\DeclareMathOperator{\vol}{vol}

\newcommand{\bm}{\mathcal{M}}

\newcommand{\R}{\mathbb{R}}

\newcommand{\Q}{\mathbb{Q}}

\let\epsilon=\varepsilon

\begin{document}

\title[A recursion for the volume of moduli spaces]{A recursion for the volume of the moduli space of hyperbolic spheres}
\author{\sc Michele Ancona and Damien Gayet}
\maketitle
\begin{abstract}
We prove the existence of a non-linear recursive relation for the volume of the moduli space of hyperbolic spheres with conical points or geodesic boundaries. This relation generalizes a result by Zograf~\cite{zograf1993weil}, where the same was derived for cusps. 
\end{abstract}
\tableofcontents
\section{Introduction}
The goal of this paper is to provide a recursion formula à la Zograf for the Weil--Petersson volume of the moduli space of genus zero hyperbolic surfaces with conical singularities (Theorem \ref{recutheta}). The main application is a new computation of the Weil--Petersson volume of the hyperelliptic locus $\mathcal{H}_g$ in $\mathcal{M}_g$. We also provide an analogous formula for the Weil--Petersson volume of the moduli space of genus zero hyperbolic surfaces with geodesic boundaries (Theorem \ref{recu}). In this case, we derive a differential equation satisfied by the generating series for the volumes of the moduli spaces of genus zero hyperbolic surfaces with one geodesic boundary and $n-1$ cusps (Theorem \ref{h1}). When the length of the geodesic boundary tends to zero, this differential equation specializes to the one found by Zograf and  Kaufmann--Manin--Zagier \cite{kaufmann1996higher}.%
\subsection{Spheres with conical singularities}

Given a Riemann surface $X$ and a point $x\in X$, we say that a metric on $X$ has a conical singularity of angle $\theta\in \left]0,2\pi \right]$ at $x$ if there exist holomorphic coordinates around $x$ in which the metric is given by  $$\frac{\rho(z)|\mathrm{d}z|^2}{|z|^{2a}},$$ where $\theta=2\pi(1-a)$ and $\rho$ is a positive function. The case $\theta=0$ corresponds to a cusp; in that case, one can find holomorphic coordinates around $x$ in which the metric is given by $$\frac{\rho(z)|\mathrm{d}z|^2}{|z|^{2}\log^2(1/|z|^2)}.$$ 

Given a closed orientable surface $X$ with $n$ marked points $x_1,\dots,x_n$, and given $n$ angles $\theta=(\theta_1, \cdots, \theta_n)\in [0,2\pi]^n$, the necessary and sufficient condition for $X$ to admit a hyperbolic metric with conical singularities of angles $\theta_1,\dots,\theta_n$ at the points $x_1,\dots,x_n$ is given by the Gauss--Bonnet formula; see~\cite{mcowen1988point, troyanov1991prescribing}. In the genus $0$ case, the condition becomes:
\begin{equation}\label{cond}
4\pi-\sum_{i=1}^n(2\pi-\theta_i)<0.
\end{equation}
For any $\theta=(\theta_1,\dots,\theta_n)\in [0,2\pi]^n$ satisfying condition~\eqref{cond}, we denote by $\mathcal M_{0,n}(i\theta)$ the moduli space of genus $0$ hyperbolic surfaces with $n$ conical points of angles $\theta$. Note that $\mathcal M_{0,n}(i0)$ identifies with the standard moduli space $\mathcal M_{0,n}$ of marked genus $0$ Riemann surfaces.

The moduli space $\mathcal M_{0,n}(i\theta)$ is endowed with a Weil--Petersson-type symplectic form $\omega_{\mathrm{WP}}(i\theta)$; see~\cite{schumacher2011weil}. When $\theta=(0,\dots,0)$, the form $\omega_{\mathrm{WP}}(i0)$ coincides with the standard Weil--Petersson symplectic form $\omega_{\mathrm{WP}}$ on $\mathcal M_{0,n}$.
Let 
\begin{equation}\label{def volume}
\vol_{\omega_{\mathrm{WP}}(i\theta)}(\mathcal M_{0,n}(i\theta))= \frac{1}{(n-3)!}\int_{\mathcal M_{0,n}(i\theta)}\omega_{\mathrm{WP}}(i\theta)^{n-3}
\end{equation}
be the volume of the moduli space of spheres with $n$ conical points of angle $\theta$.
Such volumes are finite have been studied in several occasions (see for example~\cite{wolpert1983homology,do2009weil,schumacher2011weil,anagnostou2022volumes,turiaci2023}). In \cite{anagnostou2023weil}, the dependence of such volumes on $\theta \in [0,2\pi]^n$ is studied; in particular, it is shown that the volume is a piecewise polynomial function of $\theta$.

Our first main theorem, proved using the techniques of \cite{zograf1993weil}, is a recursion formula for computing the volumes \eqref{def volume}.
Before stating our result, it is convenient to define the volume also in the case when Condition \eqref{cond} is not satisfied. For this, remark that $\mathcal{M}_{0,3}(i\theta)$ is either a point (when Condition \eqref{cond} is satisfied), either empty. We nevertheless define  $\vol_{\omega_{\mathrm{WP}}(i\theta)}(\mathcal M_{0,3}(i\theta))=1$ for any $\theta=(\theta_1,\theta_2,\theta_3)\in [0,2\pi]^3$, so that the volume function is continuous on $[0,2\pi]^3$.
Similarly, one can show that, if $n\geq 4$, then $\vol_{\omega_{\mathrm{WP}}(i\theta)}(\mathcal M_{0,n}(i\theta))\to 0$ when $4\pi-\sum_{i=1}^n(2\pi-\theta_i)\to 0$. In this case we define $\vol_{\omega_{\mathrm{WP}}(i\theta)}(\mathcal M_{0,n}(i\theta))=0$ for any $\theta$ not satisfying Condition \eqref{cond}, when $n\geq 4$, so that the volume function is also continuous on $[0,2\pi]^n$.

Define $V_n(i\theta)$ as
$$ \vol_{\omega_{\mathrm{WP}}(i\theta)}(\mathcal{M}_{0,n}(i\theta)) = 
\frac{(2\pi^2)^{n-3}}{(n-3)!}V_n(i\theta).$$

\begin{theorem}\label{recutheta}
	For $n\geq 4$ and $\theta\in [0,2\pi]^n$ satisfying condition~\eqref{cond},
	\begin{align*} V_n(i\theta) = \frac12\sum_{i=1}^{n-3}
	\frac{1}{n-1}i(n-i-2)
	{n-4 \choose i-1}
	\sum_{\substack{I\subset\{1, \cdots, n\}\\|I|=i+1}}
	V_{i+2}(i\theta_{|I},0)V_{n-i}(i\theta_{|I^c},0)\\
	- \frac{1}{4\pi^2}
	\sum_{i=1}^{n-3}
	{n-4 \choose i-1}{n-1\choose 2}^{-1}
	\sum_{\substack{j,k,\ell\in  \{1, \cdots, n\}\\
			j\neq k, k\neq \ell, j\neq \ell
	}}
	\theta_j^2
	\sum_{\substack{I\subset\{1, \cdots, n\} \\|I|=i+1 \\ j\in I, k\in I^c, \ell\in I^c}}
	V_{i+2}(i\theta_{|I},0)V_{n-i}(i\theta_{|I^c},0)\\
	+\frac{1}{4\pi^2} \sum_{I\in\mathcal S_2(\theta)}(2\pi-\sum_{i\in I}(2\pi-\theta_i))^2V_{n-1}(i\theta_{|I^c},0).
	\end{align*}%
	Here, $S_2(\theta)$ is the set of subsets $I$ of $\{1,\dots,n\}$ verifying
	\begin{itemize}
		\item $|I|=2$;
		\item $2\pi-\sum_{i\in I}(2\pi-\theta_i)\geq 0$.
	\end{itemize} 
\end{theorem}

The main motivation behind Theorem \ref{recutheta} is the computation of the volume of the moduli space of genus $g$ hyperelliptic surfaces $\mathcal H_g\subset \bm_{g}$.
Indeed, by~\cite[Section 4, p.~234]{faber2000logarithmic},
\begin{equation}
\vol_{\omega_{\mathrm{WP}}} \mathcal H_g =
\frac{1}{(2g+2)!}
\vol \bm_{0,2g+2}(i\pi, \cdots, i\pi).
\end{equation}
This raises a natural question: is it possible to find a simple asymptotic formula for $\vol(\mathcal{H}_g)$ as $g\to+\infty$?

\subsection{Spheres with boundaries}
Let now $L=(L_1,\cdots, L_n)\in \R_+^n$ be an $n$-tuple of non-negative real numbers. We denote by
$\bm_{g,n}(L)$ the moduli space of genus $g$ hyperbolic surfaces with $n$ geodesic boundaries of lengths $L_1,\cdots, L_n$. When $L=(0,\dots, 0)$, the moduli space $\bm_{g,n}(0)$ can be identified with the moduli space $ \bm_{g,n}$ of hyperbolic surfaces with $n$ marked points (geometrically, the boundaries degenerate into cusps).

By work of Goldman \cite{goldman1984symplectic}, the moduli space $\bm_{g,n}(L)$ carries a natural Weil--Petersson symplectic form $\omega_{\mathrm{WP}}(L)$. The associated volume, called the Weil--Petersson volume, is finite. When $L=(0,\dots,0)$, the form $\omega_{\mathrm{WP}}(0)$ coincides with the standard Weil--Petersson symplectic form $\omega_{\mathrm{WP}}$ on $\mathcal M_{0,n}$.
Weil--Petersson volumes have been estimated and computed using a wide range of methods, from algebraic to hyperbolic geometry, and many recursive formulas have been found (see for instance~\cite{mirzakhani2007simple,liu2009recursion,mirzakhani13,zograf1993weil,kaufmann1996higher, manin2000invertible,do2009weil, mirzakhani2015towards,mirzakhani2019lengths,hide2025large}). 

The next theorem we present is a recursion formula for the Weil--Petersson volume of the moduli space of genus $0$ curves with boundary components.
To fix notation, let
$$ \vol_{\omega_{\mathrm{WP}}(L)}(\mathcal M_{0,n}(L))= \frac{1}{(n-3)!}\int_{\mathcal M_{0,n}(L)}\omega_{\mathrm{WP}}(L)^{n-3}$$
be the volume of the moduli space of spheres with $n$ geodesic boundaries of lengths $L$, and define $V_n(L)$ as
$$ \vol_{\omega_{\mathrm{WP}}(L)}(\mathcal M_{0,n}(L)) = 
\frac{(2\pi^2)^{n-3}}{(n-3)!}V_n(L).$$
Note that $\displaystyle V_n:=V_n(0) = \frac{(n-3)!}{(2\pi^2)^{n-3}}\vol_{\omega_{\mathrm{WP}}} \bm_{0,n}.$ 
\begin{theorem}\label{recu}For any $L\in \R_+^3$, set $V_3(L) = 1$. 
	For $n\geq 3$ and $L\in \R_+^n$, 	we have
	\begin{eqnarray*} V_n(L) &=& \frac12\sum_{i=1}^{n-3}
		\frac{1}{n-1}i(n-i-2)
		{n-4 \choose i-1}
		\sum_{\substack{I\subset \{1, \cdots, n\}\\ | I| = i+1}}
		V_{i+2}(L_{|I},0)V_{n-i}(L_{|I^c},0)\\
		&& 
		+\frac{1}{4\pi^2}
		\sum_{i=1}^{n-3}
		{n-4 \choose i-1}{n-1\choose 2}^{-1}\sum_{\substack{j,k,\ell\in  \{1, \cdots n\}\\
				j\neq k, k\neq \ell, j\neq \ell
		}}
		L_j^2
		\sum_{\substack{I\subset \{1, \cdots, n\}\\|I|=i+1 \\ j\in I,  k\in I^c, \ell\in I^c}}
		V_{i+2}(L_{|I},0)V_{n-i}(L_{|I^c},0).
	\end{eqnarray*}
\end{theorem}

For any $q\in \{0, \cdots, n\}$ and $L \in \R_+$, let
\begin{equation}\label{vqn}
 V_{n}^q(L)= V_n (\underbrace{L, \cdots, L}_{q\text{ times}}, \underbrace{0, \cdots, 0)}_{n-q\text{ times}}. 
 \end{equation}
Notice that $V^0_n(L) = V_n(0,\cdots, 0) = V_n$. 
	\begin{corollary}\label{coror}For any $L\in \R_+^3$ and any $0\leq q \leq 3$, set $V^q_3(L) = 1$. 
		For any $n\geq 4$ and $q\in \{0, \cdots, n\},$ we have
		\begin{eqnarray*}
			V^q_n(L) & =& \frac12
			\sum_{i=1}^{n-3}
			{n-4 \choose i-1}
			\sum_{p=\max(0, i+1+q-n)}^{\min(q,i+1)} 
			\\&&
			{q \choose p}{n-q \choose i+1-p}
			V_{i+2}^{p}(L)V_{n-i}^{q-p}(L)
			\left(	\frac{1}{n-1}i(n-i-2)	+ \frac{p}2\frac{L^2}{2\pi^2}\right).
		\end{eqnarray*}
	\end{corollary}
In the particular case where $L=0$, using Vandermonde identity, we recover a former result 
by Zograf:
\begin{corollary}~\cite{zograf1993weil}
	For all $n\geq 4$, 
	$$ 
	V_n = \frac12 \sum_{i=1}^{n-3} \frac{i(n-i-2)}{n-1}
	{n-4 \choose {i-1}}{n \choose {i+1}} V_{i+2}V_{n-i}.
	$$
\end{corollary}
In~\cite{zograf1993weil} the author used $\frac{1}2\omega_{\mathrm{WP}}$ instead of the classical $\omega_{\mathrm{WP}}$. However, we defined $V_n(L)$ so that  $V_n:=V_n(0)$ is exactly equal to the quantity denoted by $v_n$ in~\cite{zograf1993weil}.

\subsection{A differential equation}

For any $L\in \R_+$, let 
\begin{equation}\label{hache1} h_1(x)= \sum_{n=3}^\infty \frac{V^1_n(L) }{(n-1)!(n-3)!}x^{n-1},
\end{equation}
be the generating function of the sequence $(V^1_n(L))_{n\geq 3}$,
where $V^1_n(L)$ has been defined by~(\ref{vqn}).
Similarly, let \begin{equation}\label{hache} h(x)=  \sum_{n=3}^\infty
\frac{V_n}{(n-1)!(n-3)!} x^{n-1}
\end{equation}
be the generating function of $(V_n)_n$, see~\cite[(0.5)]{kaufmann1996higher}. The recursion given by Corollary~\ref{coror} implies the following theorem. 
\begin{theorem}\label{h1}For any $L\in \R_+$, 
	the function $h_1$ satisfies the differential equation
	\beqe 
	xh_1''-h_1' & =&
	h_1''(xh'-h)
	+ \frac{L^2}{4\pi^2}(
	h_1''h +h_1'h')
	\eeqe
\end{theorem}

If $L=0$, one has $h_1=h$, so that Theorem~\ref{h1} implies the known corollary:
\begin{corollary}\label{mazo}~\cite[(0.6)]{kaufmann1996higher}
	The function $h$ satisfies the nonlinear differential equation	$$xh''-h'=h''(xh'-h) .$$
\end{corollary}

\section{Proofs of the results}

\subsection{Proof of the recursion for boundaries}

The moduli space $\bm_{0,n}(L)$ admits a Deligne--Mumford compactification $\overline{\bm}_{0,n}(L)$ and the Weil--Petersson form $\omega_{\mathrm{WP}}(L)$  extends to a closed form, still denoted by $\bm_{0,n}(L)$, on  $\overline{\bm}_{0,n}(L)$. In particular, since the boundary $\overline{\bm}_{0,n}(L)\setminus \bm_{0,n}(L)$ has measure $0$, the Weil--Petersson volume can be computed as an integral over $\overline{\bm}_{0,n}(L)$, with the advantage that $\overline{\bm}_{0,n}(L)$ is a compact oriented manifold admitting a fundamental class, and therefore one can use the homology/cohomology pairing to compute the volume. To be able to do this, one needs to understand the cohomology class of $\omega_{\mathrm{WP}}(L)$.

Mirzakhani \cite{mirzakhani2007weil} constructed a homeomorphism from the moduli space $\overline{\mathcal{M}}_{0,n}(L)$ to $\overline{\mathcal{M}}_{0,n}$. 
One can then transport the (extension of the) Weil--Petersson from $\overline{\mathcal{M}}_{0,n}(L)$ to $\overline{\mathcal{M}}_{0,n}$. 
The next results computes the cohomology class of $\omega_{\mathrm{WP}}(L)$ inside $H^2(\overline{\bm}_{0,n},\R).$
\begin{theorem}\label{mimi}\cite[Theorem 3.2]{mirzakhani2007weil} Let $n\geq 3$ and $L\in \R_+^n$. Then,
	$$[\omega_{\mathrm{WP}}(L)] = [\omega_{\mathrm{WP}}] +\frac{1}2\sum_{j=1}^n  L_j^2\psi_j \in H^2(\overline{\bm}_{0,n},\R).$$
\end{theorem}
Recall that for any $i\in {1, \cdots, n}$, one defines $\psi_i\in H^2(\overline{\bm}_{0,n},\Q)$ to be the first Chern class on $\overline{\bm}_{0,n}$ of the holomorphic line bundle $\mathcal L_i$, whose fiber over a point $[C,x_1,\dots,x_n] \in \overline{\bm}_{0,n}$
is the cotangent line at $x_i$ to the
curve $[C,x_1,\dots,x_n]$.

\begin{definition}\label{dedef}
	Let $n\geq 3$. 
For any subset $I\subset \{1, \cdots, n\}$, 
let $\Delta_I$ be the divisor in $\overline{\mathcal M}_{0,n}$ whose generic point represents a one-nodal genus $0$ curve, with the $I$-marked points on one irreducible components and the $I^{c}$-marked points on the other one. 
\end{definition}
Notice that $\Delta_I\cong \overline{\mathcal M}_{I\cup\{n+1\}}\times \overline{\mathcal M}_{I^c\cup\{n+2\}}$, where $\overline{\mathcal M}_{I\cup\{n+1\}}$ stands for $\overline{\mathcal M}_{0,I\cup\{n+1\}}$ (and similarly for $\overline{\mathcal M}_{I^c\cup\{n+2\}}$).
In the isomorphism, the marked points $n+1$ and $n+2$ are identified to create the node attaching the rational curve containing the $I$-marked points to the rational curve containing the $I^c$ ones.

We now state  classical computations of Poincaré dual of the Weil--Petersson symplectic form and of the $\psi$-classes.
\begin{proposition}\label{zozo}\cite[p. 371]{zograf1993weil}
Let $n\geq 3$. The Poincaré dual of $[\omega_{\mathrm{WP}}]$ equals 
\begin{equation}\nonumber	
\mathrm{PD}([\omega_{\mathrm{WP}}]) = \frac{\pi^2}{(n-1)}\sum_{i=1}^{n-3}i(n-i-2) 
\sum_{\substack{I\subset \{1, \cdots, n\}\\|I| = i+1}} \Delta_I \in H^2(\overline{\bm}_{0,n},\R).
\end{equation}
		\end{proposition}
\begin{remark} In~\cite{zograf1993weil}, the Weil--Petersson form is one half times $\omega_{\mathrm{WP}}$, hence the factor $2$ difference with his expression.
\end{remark}

\begin{lemma}\label{psik} \cite[Proposition 2.13]{zvonkine2012introduction}
Let $n\geq 3$. For any $j\in \{1, \cdots, n\}$, the Poincaré dual of $\psi_i$ equals 
$$\mathrm{PD}(\psi_j) = 
{n-1 \choose 2}^{-1}
\sum_{\substack{k,\ell\in  \{1, \cdots n\}\\
		j\neq k, k\neq \ell, j\neq \ell
}}
\sum_{\substack{I\subset \{1, \cdots, n\}\\   j\in I,k\in I^c, \ell\in I^c}}\Delta_I
\in H^2(\overline{\bm}_{0,n},\R).
$$  
\end{lemma}

\begin{corollary}
\label{coroWP}Let $n\geq 3$ and $L\in \R_+^n$. Then,
\begin{eqnarray*}
\mathrm{PD}[\omega_{\mathrm{WP}}(L)] &=& 
\frac{\pi^2}{(n-1)}\sum_{i=1}^{n-3}i(n-i-2) 
\sum_{\substack{I\subset \{1, \cdots, n\}\\|I| = i+1}} \Delta_I
\\ &&
 +\frac{1}2\sum_{j=1}^n  L_j^2
 {n-1 \choose 2}^{-1}
 \sum_{\substack{k,\ell\in  \{1, \cdots n\}\\
 		j\neq k, k\neq \ell, j\neq \ell
 }}
 \sum_{\substack{I\subset \{1, \cdots, n\}\\   j\in I,k\in I^c, \ell\in I^c}}\Delta_I \in H^2(\overline{\bm}_{0,n},\R).
\end{eqnarray*}
\end{corollary}
\begin{preuve}
	This is a direct consequence of Theorem~\ref{mimi}, Proposition~\ref{zozo} and Lemma~\ref{psik}.
	\end{preuve}
The following three lemmas explains how the divisors $\Delta_I$'s intersect each others. More details on the intersection ring of $\overline{\bm}_{0,n}$ can be found in \cite{Keel1992}.
\begin{lemma}[Geometric intersection of boundary divisors]\label{geometric} Let $I$ and $J$ be subsets of $\{1,\dots n\}$.  Then $\Delta_I\cap \Delta_J\neq\emptyset$ if and only if $I\subseteq J$ or $I^c\subseteq J$ or $I\subseteq J^c$ or $I^c\subseteq J^c$.
\end{lemma}

\begin{lemma}[Transverse intersection of boundary divisors]\label{transverse} Let $I$ and $J$ be subsets of $\{1,\dots n\}$. Then,
	$$\Delta_I\cap \Delta_J=
	\left\lbrace\begin{array}{ll} \Delta_{I\cup\{n+1\}}\times\overline{\mathcal{M}}_{J^c\cup\{n+2\}}  & \text{ if } I\subset J\\
	\Delta_{I^c\cup\{n+1\}}\times\overline{\mathcal{M}}_{J^c\cup\{n+2\}}  & \text{ if } I^c\subset J\\\
	\overline{\mathcal{M}}_{J\cup\{n+1\}}\times \Delta_{I\cup\{n+2\}} & \text{ if } I\subset J^c\\\
	\overline{\mathcal{M}}_{J\cup\{n+1\}}\times \Delta_{I^c\cup\{n+2\}} & \text{ if } I^c\subset J^c.
	\end{array}\right.
	$$
	Here there is a slight, classic abuse of notation, since the divisors live in different moduli spaces.
\end{lemma} 

For any subset $I\subset \{1, \cdots, n\}$, 
let $$\iota_I : \overline{\mathcal M}_{I\cup\{n+1\}}\times \overline{\mathcal M}_{I^c\cup\{n+2\}}\hookrightarrow \overline{\mathcal M}_{0,n}.$$ This induces the restriction map $$\iota_I^*:H^2( \overline{\mathcal M}_{0,n})\rightarrow H^2(\overline{\mathcal M}_{I\cup\{n+1\}}\times \overline{\mathcal M}_{I^c\cup\{n+2\}}).$$ 
By K\"unneth decomposition, we have 
	\begin{eqnarray*}
		H^2(\overline{\mathcal M}_{I\cup\{n+1\}}\times \overline{\mathcal M}_{I^c\cup\{n+2\}})&\cong &H^2(\overline{\mathcal M}_{I\cup\{n+1\}})\oplus H^2( \overline{\mathcal M}_{I^c\cup\{n+2\}})
		\oplus 
		\\&&
		H^1(\overline{\mathcal M}_{I\cup\{n+1\}})\otimes H^1(\overline{\mathcal M}_{I^c\cup\{n+2\}})
	\end{eqnarray*}
By a standard abuse of notation, we denote $\psi_{n+1}$ and $\psi_{n+1}$ the classes in 	$H^2(\overline{\mathcal M}_{I\cup\{n+1\}}\times \overline{\mathcal M}_{I^c\cup\{n+2\}})$ induced by the corresponding $\psi$-classes in $H^2(\overline{\mathcal M}_{I\cup\{n+1\}})$ and $H^2( \overline{\mathcal M}_{I^c\cup\{n+2\}})$ under K\"unneth decomposition.
\begin{lemma}[Auto-intersection of boundary divisors]\cite[Theorem 3.22]{zvonkine2012introduction}\label{auto} Let $I$ be a subset of $\{1,\dots n\}$.  Then,
	$\iota_I^*\Delta_I=-\psi_{n+1}-\psi_{n+2}$.
\end{lemma}

\begin{lemma}[Restriction of the Weil--Petersson form to a boundary divisor]\label{boundary} 
	Let $I\subset \{1, \cdots, n\}$. Then,
	$$\iota_I^*[\omega_{\mathrm{WP}}(L)]= [\omega_{\mathrm{WP}}(L_{|I},0)]+[\omega_{\mathrm{WP}}(L_{|I^c},0)].$$
	The same holds for $\omega_{\mathrm{WP}}(i\theta)$.
\end{lemma}
As before, the classes $[\omega_{\mathrm{WP}}(L_{|I},0)]$ and $[\omega_{\mathrm{WP}}(L_{|I^c},0)]$ live in 	$H^2(\overline{\mathcal M}_{I\cup\{n+1\}}\times \overline{\mathcal M}_{I^c\cup\{n+2\}})$ thanks to  K\"unneth decomposition.

We are now ready to prove Theorem~\ref{recu}. Recall that $$ \vol_{\omega_{\mathrm{WP}}(L)}(\overline{\mathcal M}_{0,n}(L))= \frac{1}{(n-3)!}\int_{\overline{\mathcal M}_{0,n}}\omega_{\mathrm{WP}}(L)^{n-3}$$
denotes the volume of the moduli space of spheres with $n$ geodesic boundaries of lengths $L$, and  $V_n(L)$ is defined by
$$ \vol_{\omega_{\mathrm{WP}}(L)}(\overline{\mathcal M}_{0,n}(L)) = 
\frac{(2\pi^2)^{n-3}}{(n-3)!}V_n(L).$$

\begin{preuve}[ of Theorem~\ref{recu}] We follow the ideas of ~\cite{zograf1993weil}.

Let us start  by computing the Weil--Petersson volume of the divisor $\Delta_I$. For this recall that $$\Delta_I\cong \overline{\mathcal M}_{0,I\cup\{n+1\}}\times \overline{\mathcal M}_{0,I^c\cup\{n+2\}}$$ and also recall Lemma \ref{boundary}.
For any $i\in \{1, \cdots, n\}$ and $I\subset\{1, \cdots, n\}$, with $|I|=i+1$,
\begin{eqnarray}\nonumber
	[\omega_{\mathrm{WP}}(L)]^{n-4}
	\cap \Delta_I
	&=& 
	\left(
	[\omega_{\mathrm{WP}}(L_{|I},0)]
	+
	[\omega_{\mathrm{WP}}(L_{|I^c},0)]
	\right)^{n-4}\cap  \overline{\mathcal M}_{0,I\cup\{n+1\}}\times \overline{\mathcal M}_{0,I^c\cup\{n+2\}}\\
	& = & \label{urbex} {n-4 \choose i-1}(2\pi^2)^{n-4}
	V_{i+2}(L_{|I},0)V_{n-i}(L_{|I^c},0).
\end{eqnarray}

Let us now write
\begin{equation}\label{grigri}
\int_{\overline{\mathcal M}_{0,n}}\left(\omega_{\mathrm{WP}}(L)\right)^{n-3}
= [\omega_{\mathrm{WP}}(L)]^{n-4}\cdot \mathrm{PD}([\omega_{\mathrm{WP}}(L)]).
\end{equation}
By~(\ref{grigri}), (\ref{urbex}) and Corollary~\ref{coroWP}, we obtain
\begin{eqnarray*}
	(2\pi^2)^{n-3} V_n(L)&=& 
	(2\pi^2)^{n-4} 
	\sum_{i=1}^{n-3}
\frac{\pi^2}{(n-1)}i(n-i-2) 
	{n-4 \choose i-1}
	\sum_{\substack{I\subset\{1, \cdots, n\}\\|I|=i+1}}
	V_{i+2}(L_{|I},0)V_{n-i}(L_{|I^c},0)\\
	&& + (2\pi^2)^{n-4} \frac{1}2
	\sum_{i=1}^{n-3}
	{n-4 \choose i-1}
	{n-1 \choose 2}^{-1}\\
	&&
	\sum_{\substack{k,\ell\in  \{1, \cdots n\}\\
			j\neq k, k\neq \ell, j\neq \ell
	}}
	\sum_{\substack{I\subset \{1, \cdots, n\}\\   j\in I,k\in I^c, \ell\in I^c}}
	L_j^2 
	V_{i+2}(L_{|I},0)V_{n-i}(L_{|I^c},0),
\end{eqnarray*}
hence the result.
\end{preuve}

\begin{preuve}[ of Corollary~\ref{coror}]
	Writing $I_q := I\cap \{1, \cdots, q\}$ (if $q=0$, $I_q =\emptyset$),
	by Theorem~\ref{recu},
	\begin{eqnarray*} V_n^q(L) &=&\frac12\sum_{i=1}^{n-3}
		\frac{1}{n-1}i(n-i-2)
		{n-4 \choose i-1}
		\sum_{|I|=i+1}
		V_{i+2}^{|I_q|}(L)V_{n-i}^{|(I^c)_q|}(L)\\
		&& 
		+ \frac{	L^2}{4\pi^2}
		\sum_{i=1}^{n-3}
		{n-4 \choose i-1}	
		\sum_{|I|=i+1}
		\sum_{\substack{j\in I_q\\ k(j)\notin I,\ell(j)\notin I}} 
		V_{i+2}^{|I_q|}(L)V_{n-i}^{|(I^c)_q|}(L)
		.
	\end{eqnarray*}
Here, $k(j)$ and $\ell(j)$ are fixed two distinct points of $\{1, \cdots, n\}\setminus\{j\}$.
	Since 
	\begin{eqnarray*} 
		\sum_{|I|=i+1}
		\sum_{\substack{j\in I_q\\ k(j)\notin I,\ell(j)\notin I}} 
		V_{i+2}^{|I_q|}(L)V_{n-i}^{|(I^c)_q|}(L) &=& 
		\sum_{|I|=i+1}|I_q|
		V_{i+2}^{|I_q|}(L)V_{n-i}^{|(I^c)_q|}(L)  \\
		&=& 
		\sum_{p=\max(0, i+1+q-n)}^{\min(q,i+1)}
		p
		\sum_{|I|=i+1, |I_q|=p}
		V_{i+2}^{p}(L)V_{n-i}^{q-p}(L) ,
	\end{eqnarray*}
we obtain that for every $q\in \{0, \cdots, n\}$,
	\begin{eqnarray*} 
	V_n^q(L) &=& 
	\nonumber\frac12
	\sum_{i=1}^{n-3}
	\frac{1}{n-1}i(n-i-2)
	{n-4 \choose i-1}
	\sum_{p=\max(0, i+1+q-n)}^{\min(q,i+1)} {q \choose p}{n-q \choose i+1-p}
	V_{i+2}^{p}(L)V_{n-i}^{q-p}(L)\\
	&& \nonumber
	+ \frac{	L^2}{4\pi^2}
	\sum_{i=1}^{n-3}
	{n-4 \choose i-1}	 \sum_{p=\max(0, i+1+q-n)}^{\min(q,i+1)}
	p	 {q \choose p}{n-q \choose i+1-p}
	V_{i+2}^{p}(L)V_{n-i}^{q-p}(L), 
	\end{eqnarray*}	
hence the result.
\end{preuve}
\begin{remark}Note that 
$$ V^n_n(L) = \frac12
\sum_{i=1}^{n-3}
{n-4 \choose i-1}
 {n \choose i+1}
V_{i+2}^{i+1}(L)V_{n-i}^{n-i-1}(L)
\left(\frac{1}{n-1}i(n-i-2)	+ (i+1)\frac{L^2}{2\pi^2}\right).$$
\end{remark}

\subsection{Proof of the recursion for conical points}\label{zpta}
Let $\theta=(\theta_1,\dots,\theta_n)\in [0,2\pi]^n$ be angles satisfying Condition \ref{cond}. The moduli space $\bm_{0,n}(i\theta)$ admits a compactification $\overline{\bm}_{0,n}(i\theta)$, called the Hassett compactification \cite{hassett2003moduli}. The form $\omega_{\mathrm{WP}}(i\theta)$ extends on the Hassett compactification to a closed $(1,1)$-form, still denoted by $\omega_{\mathrm{WP}}(i\theta)$, see \cite{schumacher2011weil,anagnostou2022volumes}.

Remark that there is a birational map $\overline{\mathcal M}_{0,n}\rightarrow \overline{\mathcal M}_{n}(i\theta)$ from Deligne--Mumford compactification $\overline{\mathcal M}_{0,n}$ to the Hassett compactification, contracting some boundary divisors of $\overline{\mathcal M}_{0,n}$. 
Thus, by pulling-back $\omega_{\mathrm{WP}}(i\theta)$ to the Deligne--Mumford compactification we obtain a closed $(1,1)$-form on $\overline{\mathcal M}_{0,n}$, still denoted by $\omega_{\mathrm{WP}}(i\theta)$. As the boundary divisors added in the compactification have measure zero, the Weil--Petersson volume we want to compute equals

$$ \vol_{\omega_{\mathrm{WP}}(i\theta)}(\mathcal M_{0,n}(i\theta))= \frac{1}{(n-3)!}\int_{\overline{\mathcal M}_{0,n}}\omega_{\mathrm{WP}}(i\theta)^{n-3}$$
and can be computed by means of intersection theory on $\overline{\mathcal M}_{0,n}$.


The following lemma is the conical version of Theorem~\ref{mimi}.
\begin{lemma}[Cohomology class of Weil--Petersson form]\cite[Proposition 2.7]{anagnostou2023weil}\label{cohomology}
The cohomology class of $\omega_{\mathrm{WP}}(i\theta)$ equals
$$[\omega_{\mathrm{WP}}(i\theta)]=[\omega_{\mathrm{WP}}]-\frac{1}{2}\sum_{i=1}^n\theta_i^2\psi_i+\frac{1}{2}\sum_{I\in\mathcal S(\theta)}(2\pi-\sum_{i\in I}(2\pi-\theta_i))^2\Delta_I
\in H^2(\overline{\bm}_{0,n}, \R).$$
Here,  $S(\theta)$ the set of subsets $I$ of $\{1,\dots,n\}$ verifying
\begin{itemize}
	\item $|I|\geq 2$;
	\item $2\pi-\sum_{i\in I}(2\pi-\theta_i)\geq 0$.
\end{itemize} 
\end{lemma} 
Recall that $\Delta_I$ is defined in  Definition~\ref{dedef}.


\begin{preuve}[ of Theorem~\ref{recutheta}]
	The proof is similar to the one of Theorem~\ref{recu}, replacing $L$ by $i\theta$, and using Lemma \ref{cohomology} instead of Theorem~\ref{mimi}.
Again, for any $i\in \{1, \cdots, n\}$ and $I\subset\{1, \cdots, n\}$, $|I|=i+1$,
\begin{eqnarray*}
	[\omega_{\mathrm{WP}}(i\theta)]^{n-4}
	\cap \Delta_I
	& = & {n-4 \choose i-1}(2\pi^2)^{n-4}
	V_{i+2}(i\theta_{|I},0)V_{n-i}(i\theta_{|I^c},0).
\end{eqnarray*}
The latter implies that that for $I\in\mathcal S(\theta)$ with $|I|>2$, one has $[\omega_{\mathrm{WP}}(i\theta)]^{n-4}
	\cap \Delta_I=0$. Indeed, $\bm_{0,I\cup\{n+1\}}(i\theta_{|I},0)=\emptyset$ because $(\theta_{|I},0)$ does not satisfy condition~\eqref{cond}. In particular,  $V_{i+2}(i\theta_{|I},0)=0$.
	
Moreover, one can check directly from the explicit computation of the cohomology class of $\omega_{\mathrm{WP}}(i\theta)$ (see Lemma \ref{cohomology}) and from Lemmas \ref{geometric},  \ref{transverse} and \ref{auto}, that 
	$$|I|=2 \Rightarrow 
	\left\lbrace
	\begin{array}{l}
	\Delta_I\cong  \overline{\mathcal M}_{0,I^c\cup\{n+2\} }\\
	\left[	\omega_{\mathrm{WP}}(i\theta) \right]
	\cap \Delta_I=[\omega_{\mathrm{WP}}(i\theta_{|I^c},0)]\\
	\left[\omega_{\mathrm{WP}}(i\theta)\right]^{n-4}
	\cap \Delta_I= (2\pi^2)^{n-4} V_{n-1}(i\theta_{|I^c},0).
	\end{array}\right.$$ 
	This justifies  the initial condition  $V_3(i\theta) = 1$ for any $\theta\in [0,2\pi]^3$.

We then obtain
\begin{eqnarray*}
V_n (i\theta )& =& \frac12 \sum_{i=1}^{n-3}
\frac{1}{n-1}i(n-i-2)
 {n-4 \choose i-1}
 \sum_{|I|=i+1}
 V_{i+2}(i\theta_{|I},0)V_{n-i}(i\theta_{|I^c},0)\\
 &&
 - \frac{1}{4\pi^2}
 \sum_{i=1}^{n-3}
 {n-4 \choose i-1}{n-1\choose 2}^{-1}
 \sum_{\substack{j,k,\ell\in  \{1, \cdots n\}\\
 	i\neq j, i\neq j, j\neq \ell
 }}
\theta_j^2
 \sum_{\substack{|I|=i+1 \\ j\in I, k\in I^c, \ell\in I^c}}
 V_{i+2}(i\theta_{|I},0)V_{n-i}(i\theta_{|I^c},0)\\
 &&
 +\frac{1}{4\pi^2}\sum_{I\in\mathcal S_2(\theta)}(2\pi-\sum_{i\in I}(2\pi-\theta_i))^2V_{n-1}(i\theta_{|I^c},0).
\end{eqnarray*}
Hence the result.
\end{preuve}

	\subsection{Proof of the differential equation}

\ 
\begin{preuve}[ of Theorem~\ref{h1}] 
	By Corollary~\ref{coror} and after some manipulations, 
	\beqe 
	xh_1''-h_1' 
	& = &
	\frac12
	\sum_{n=6}^\infty
	x^{n-4}
	\sum_{j=3}^{n-3}
	\frac{2(n-j-1)(n-j-2)(j-2)	V_{n-j}^{1}V_{j}}{(j-1)!(j-3)!(n-j-1)!(n-j-3)!}\\
	&&
	+ \frac{L^2}{4\pi^2}
	\sum_{n=6}^\infty
	x^{n-4}
	\sum_{j=3}^{n-3}\frac{(n-3)			V_{n-j}^{1}V_{j}}{(j-3)!(j-1)!(n-j-3)!(n-j-2)!}	
	\eeqe
	so that 
	\beqe 
	xh_1''-h_1' 
	& =& 
	h_1''(xh'-h)\\
	&&
	+ \frac{L^2}{4\pi^2}
	\sum_{n=6}^\infty
	x^{n-4}
	\sum_{j=3}^{n-3}		\frac{(n-j-2)(n-j-1)	V_{n-j}^{1}V_{j}}{(j-1)!(j-3)!(n-j-1)!(n-j-3)!}	\\
	&&
	+ \frac{L^2}{4\pi^2}
	\sum_{n=6}^\infty
	x^{n-4}
	\sum_{j=3}^{n-3}		\frac{(j-1)(n-j-1)	V_{n-j}^{1}V_{j}}{(j-1)!(j-3)!(n-j-1)!(n-j-3)!}	.
	\eeqe
	Finally
	\beqe 
	xh_1''-h_1' & =& 	 
	h_1''(xh'-h)
	+ \frac{L^2}{4\pi^2}(
	h_1''h +h_1'h').
	\eeqe
	Hence the result.
\end{preuve}

\vspace{1cm}
{\bf Acknowledgments}: The authors thank Roland Bacher, Philippe Eyssidieux, Jérémy Guéré, Giovanni Inchiostro, Francis Lazarus, Tanguy Rivoal and Peter Zograf. The research leading to these results has received funding from the French Agence nationale de la ANR-20-CE40-0017 (Adyct).

\bibliographystyle{amsplain}
\bibliography{hyperelliptic.bib}

\providecommand{\bysame}{\leavevmode\hbox to3em{\hrulefill}\thinspace}
\providecommand{\MR}{\relax\ifhmode\unskip\space\fi MR }
\providecommand{\MRhref}[2]{%
  \href{http://www.ams.org/mathscinet-getitem?mr=#1}{#2}
}
\providecommand{\href}[2]{#2}
\begin{thebibliography}{10}

\bibitem{anagnostou2023weil}
Lukas Anagnostou, Scott Mullane, and Paul Norbury, \emph{Weil-{P}etersson
  volumes, stability conditions and wall-crossing}, arXiv preprint
  arXiv:2310.13281 (2023), 1--38.

\bibitem{anagnostou2022volumes}
Lukas Anagnostou and Paul Norbury, \emph{Volumes of moduli spaces of hyperbolic
  surfaces with cone points}, arXiv preprint arXiv:2212.13701 (2022), 1--40.

\bibitem{do2009weil}
Norman Do and Paul Norbury, \emph{Weil--{P}etersson volumes and cone surfaces},
  Geometriae Dedicata \textbf{141} (2009), no.~1, 93--107.

\bibitem{turiaci2023}
Lorenz Eberhardt and Gustavo~J. Turiaci, \emph{2d dilaton gravity and the
  {W}eil-{P}etersson volumes with conical defects}, arXiv preprint (2023),
  1--44.

\bibitem{faber2000logarithmic}
Carel Faber and Rahul Pandharipande, \emph{{Logarithmic series and Hodge
  integrals in the tautological ring. With an appendix by Don Zagier}},
  Michigan Mathematical Journal \textbf{48} (2000), no.~1, 215--252.

\bibitem{goldman1984symplectic}
William~M Goldman, \emph{The symplectic nature of fundamental groups of
  surfaces}, Advances in Mathematics \textbf{54} (1984), no.~2, 200--225.

\bibitem{hassett2003moduli}
Brendan Hassett, \emph{Moduli spaces of weighted pointed stable curves},
  Advances in Mathematics \textbf{173} (2003), no.~2, 316--352.

\bibitem{hide2025large}
Will Hide and Joe Thomas, \emph{{Large-$n$ asymptotics for Weil-Petersson
  volumes of moduli spaces of bordered hyperbolic surfaces}}, Communications in
  Mathematical Physics \textbf{406} (2025), no.~9, 203.

\bibitem{kaufmann1996higher}
Ralph Kaufmann, Yuri Manin, and Don Zagier, \emph{{Higher Weil-Petersson
  volumes of moduli spaces of stable n-pointed curves}}, Communications in
  mathematical physics \textbf{181} (1996), no.~3, 763--787.

\bibitem{Keel1992}
S.~Keel, \emph{Intersection theory of moduli space of stable n-pointed curves
  of genus zero}, Transactions of the American Mathematical Society
  \textbf{330} (1992), 545--574.

\bibitem{liu2009recursion}
Kefeng Liu and Hao Xu, \emph{{Recursion formulae of higher Weil--Petersson
  volumes}}, International Mathematics Research Notices \textbf{2009} (2009),
  no.~5, 835--859.

\bibitem{manin2000invertible}
Yuri~I. Manin and Peter Zograf, \emph{{Invertible cohomological field theories
  and Weil-Petersson volumes}}, Annales de l'Institut Fourier \textbf{50}
  (2000), no.~2, 519--535.

\bibitem{mcowen1988point}
Robert~C McOwen, \emph{Point singularities and conformal metrics on riemann
  surfaces}, Proceedings of the American Mathematical Society \textbf{103}
  (1988), no.~1, 222--224.

\bibitem{mirzakhani2007simple}
Maryam Mirzakhani, \emph{{Simple geodesics and Weil-Petersson volumes of moduli
  spaces of bordered Riemann surfaces}}, Inventiones mathematicae \textbf{167}
  (2007), no.~1, 179--222.

\bibitem{mirzakhani2007weil}
\bysame, \emph{Weil-{P}etersson volumes and intersection theory on the moduli
  space of curves}, Journal of the American Mathematical Society \textbf{20}
  (2007), no.~1, 1--23.

\bibitem{mirzakhani13}
\bysame, \emph{{Growth of Weil-Petersson volumes and random hyperbolic surface
  of large genus}}, {J. Differ. Geom.} \textbf{94} (2013), no.~2, 267--300.

\bibitem{mirzakhani2019lengths}
Maryam Mirzakhani and Bram Petri, \emph{Lengths of closed geodesics on random
  surfaces of large genus}, Commentarii Mathematici Helvetici \textbf{94}
  (2019), no.~4, 869--889.

\bibitem{mirzakhani2015towards}
Maryam Mirzakhani and Peter Zograf, \emph{Towards large genus asymptotics of
  intersection numbers on moduli spaces of curves}, Geometric and Functional
  Analysis \textbf{25} (2015), no.~4, 1258--1289.

\bibitem{schumacher2011weil}
Georg Schumacher and Stefano Trapani, \emph{{Weil-Petersson geometry for
  families of hyperbolic conical Riemann surfaces}}, Michigan Mathematical
  Journal \textbf{60} (2011), no.~1, 3--33.

\bibitem{troyanov1991prescribing}
Marc Troyanov, \emph{Prescribing curvature on compact surfaces with conical
  singularities}, Transactions of the American Mathematical Society
  \textbf{324} (1991), no.~2, 793--821.

\bibitem{wolpert1983homology}
Scott Wolpert, \emph{On the homology of the moduli space of stable curves},
  Annals of Mathematics \textbf{118} (1983), no.~3, 491--523.

\bibitem{zograf1993weil}
Peter Zograf, \emph{{The Weil-Petersson volume of the moduli space of punctured
  spheres}}, Contemporary Mathematics \textbf{150} (1993), 367--367.

\bibitem{zvonkine2012introduction}
Dimitri Zvonkine, \emph{An introduction to moduli spaces of curves and their
  intersection theory}, Handbook of Teichm{\"u}ller theory \textbf{3} (2012),
  667--716.

\end{thebibliography}
\vspace{12mm}

\noindent
Michele Ancona \\
Laboratoire J.A. Dieudonn\'e\\
UMR CNRS 7351\\
Universit\'e C\^ote d'Azur, Parc Valrose\\
06108 Nice, Cedex 2, France\\

\noindent
Damien Gayet\\
Institut Fourier, UMR 5582, \\Laboratoire de Mathématiques, Université Grenoble Alpes, \\Institut Universitaire de France, \\CS 40700, 38058 Grenoble cedex 9, France

\end{document}